\documentclass[11pt]{article}

\usepackage{amsthm,amsmath,mathrsfs}
\usepackage{amssymb,amsthm,amsmath}
\usepackage[dvips]{geometry}
\usepackage{enumitem}
\usepackage{amscd}
\usepackage{amsfonts}
\usepackage[latin1]{inputenc}
\usepackage[all] {xy}
\usepackage[normalem]{ulem}
\usepackage{hyperref}
\usepackage{mathtools}
\usepackage[dvips]{geometry}
\usepackage[dvips]{graphicx}
\usepackage[nottoc]{tocbibind}

\newtheorem{theorem}{Theorem}[section]

\newtheorem*{theorema}{Theorem A}
\newtheorem*{theoremb}{Theorem B}
\newtheorem*{theoremc}{Theorem C}
\newtheorem*{question1}{Question 1}
\newtheorem*{question2}{Question 2}

\newtheorem{proposition}[theorem]{Proposition}
\newtheorem{corollary}[theorem]{Corollary}
\newtheorem{lemma}[theorem]{Lemma}

\newtheorem{definition}[theorem]{Definition}

\newtheorem{remark}[theorem]{Remark}

\newcommand{\T}{\mathbb{T}}

\usepackage{ulem}

\newcommand{\comment}[1]{}
\newcommand{\Z}{\mathbb{Z}}
\newcommand{\N}{\mathbb{N}}
\newcommand{\R}{\mathbb{R}}

\newcommand{\dworld}{\mathrm{Diff}^2_{\omega}(M)}

\newcommand{\address}{{
  \bigskip
  \footnotesize

  \textsc{CNRS-Laboratoire de Math\'ematiques d'Orsay, UMR 8628, Universit\'e Paris-Sud 11, Orsay Cedex 91405, France } \par\nopagebreak
  \textsc{Instituto de Matem\'atica, Universidade Federal do Rio de Janeiro, P.O. Box 68530, 21945-970, Rio de Janeiro Brazil}\par\nopagebreak
  \textit{E-mail:} \texttt{davi.obata@math.u-psud.fr}
}}

\title{On the stable ergodicity of diffeomorphisms with dominated splitting}
\author{Davi Obata\footnote{D.O. was supported by the ERC project 692925 NUHGD.}}

\begin{document}

\maketitle
\begin{abstract}
In this paper we obtain two criteria of stable ergodicity outside the partially hyperbolic scenario. In both criteria, we use a weak form of hyperbolicity called chain-hyperbolicity. It is obtained one criterion for diffeomorphisms with dominated splitting and one criterion for weakly partially hyperbolic diffeomorphisms. As an application of one of these criteria, we obtain the $C^1$-density of stable ergodicity inside a certain class of weakly partially hyperbolic diffeomorphisms.
\end{abstract}
\setcounter{tocdepth}{1}
\tableofcontents
\section{Introduction}

Conservative dynamics appears naturally in several different areas of mathematics and physics. By a conservative dynamical system we mean a diffeomorphism of a smooth compact connected riemannian manifold $M$, that preserves a volume form $\omega$. A key property that a conservative system may have is ergodicity. 

Let $m$ be the probability measure induced by $\omega$. We say that the conservative system $f$ is \textit{ergodic} if every measurable $f$-invariant set has either full or zero $m$-measure. From the probabilistic point of view, ergodicity means that the system cannot be decomposed into smaller $f$-invariant parts. A key characterization of ergodicity is given as a consequence of the well known Birkhoff's ergodic theorem. In our setting, this can be stated as follows: $f$ is ergodic if and only if for every continuous function $\varphi:M \to R$, for $m$-almost every point $x\in M$, it holds
\begin{equation}
\label{ergodicitydef}
\displaystyle \lim_{n\to + \infty} \frac{1}{n}\sum_{j=0}^{n-1} \varphi \circ f^j(x) = \int_M \varphi dm.
\end{equation}

An important problem in the theory of dynamical systems is to know when a conservative system is ergodic. Another important question is to know when ergodicity is a robust property. Let $\mathrm{Diff}^r_{\omega}(M)$ be the space of $C^r$-diffeomorphisms of $M$ that preserves the volume form $\omega$. 

\begin{definition}
\label{stableergodicity}
A diffeomorphism $f\in \mathrm{Diff}^2_{\omega}(M)$ is {\it stably ergodic} if there exists a $C^1$-neighborhood $\mathcal{U}$ of $f$, such that any diffeomorphism $g\in \mathcal{U} \cap \mathrm{Diff}^2_{\omega}(M)$ is ergodic.
\end{definition}  

Hopf introduced an argument to prove that the geodesic flow on compact surfaces of constant negative curvature is ergodic for the Liouville measure, see \cite{hopf39}. Anosov \cite{anosov67}, Anosov and Sinai \cite{anosovsinai67} used Hopf argument to prove that every $C^2$ hyperbolic diffeomorphism is ergodic, see \cite{manebook} for the definition of hyperbolic, or Anosov, diffeomorphism. Since hyperbolicity is a $C^1$-open property, we conclude that conservative hyperbolic diffeomorphisms are stably ergodic. What about outside the hyperbolic world?

Since ergodicity is a global feature, it is natural to look for global properties of a diffeomorphism that could help to obtain ergodicity, or stable ergodicity. Among the partially hyperbolic diffeomorphisms the key global property used is the \emph{accessibility}: any two points in the manifold can be connected by a path contained in finitely many stable and unstable manifolds, see \cite{bw10} for precise definitions of partial hyperbolicity and accessibility. Pugh and Shub conjectured in \cite{pughshub97} that accessibility implies ergodicity. This conjecture remains open and one usually needs some extra assumption to conclude ergodicity. 

Most works done about the problem of stable ergodicity consider partially hyperbolic systems, see for instance \cite{brinpesin74}, \cite{gps94}, \cite{pughshub00}, \cite{bdp02} and \cite{bw10}. Not much has been done outside the partially hyperbolic scenario. 

It is known that stably ergodic diffeomorphisms must have some weaker form of hyperbolicity \cite{arbietomatheus}, called dominated splitting. We say that a diffeomorphism $f$ admits a \emph{dominated splitting} if there is a decomposition of the tangent bundle, $TM = E \oplus F$, into two non-trivial subbundles which are $Df$-invariant, such that for some $N\geq 1$, any unit vectors $v\in E(x)$ and $u\in F(x)$ verify
\[
\|Df^N(x)v\| < \frac{1}{2} \|Df^N(x)u\|. 
\]

Tahzibi builts in \cite{tahzibi04} an example of a stably ergodic diffeomorphism which is not partially hyperbolic. Another important property in ergodic theory is the following.

\begin{definition}
Let $\nu$ be an invariant measure for $f$. We say that $(f,\nu)$ is \textit{Bernoulli} if it is measurably conjugated to a Bernoulli shift. If $f$ is a diffeomorphism that preserves a smooth measure $m$, we say that $f$ is Bernoulli if $(f,m)$ is Bernoulli.
\end{definition}

We remark that the Bernoulli property is stronger than ergodicity. Switching ergodicity by the Bernoulli property in definition \ref{stableergodicity}, we obtain the definition of \textit{stably Bernoulli}.

One of the goals of this work is to find new criteria of stable ergodicity, actually of stable Bernoulli, outside the partially hyperbolic scenario. In particular, we study the consequences given by a property called \textit{chain-hyperbolicity}, see definition \ref{chainhyperbolicity}. Chain-hyperbolicity has been defined and used before in \cite{crovisier2011partial}, \cite{crovisierpujalsessential}. It can be seen as some type of topological hyperbolicity saying that $f$ ``contracts" topologically along the direction $E$, up to a certain ``scale", and $f^{-1}$ ``contracts" topologically along the direction $F$, up to a certain ``scale". Using this as the global property to study stable ergodicity, we have the following theorem. 

\begin{theorema}
\label{theorema}
Let $f\in \mathrm{Diff}^1_{\omega}(M)$. If $f$ is a chain-hyperbolic diffeomorphism for a dominated splitting $TM= E \oplus F$ and verifies
\begin{equation}
\label{hyp}
\displaystyle \int_M \log \|Df|_E\| dm <0 \textrm{ and } \int_M \log \|Df^{-1}|_F\| dm <0,
\end{equation}
then there exists a $C^1$-neighborhood $\mathcal{U}$ of $f$, such that any diffeomorphism $g\in \mathcal{U} \cap \mathrm{Diff}^2_{\omega}(M)$ is ergodic, in fact Bernoulli. In particular, any such diffeomorphism $g$ is stably Bernoulli. 
\end{theorema}

 In the setting of theorem $A$, as a consequence of (\ref{hyp}) and ergodicity, we actually obtain that $m$-almost every point has all Lyapunov exponents negative along $E$ and all positive along $F$, see section \ref{preliminaries} for the definition of Lyapunov exponent. 

A diffeomorphism  is weakly partially hyperbolic if it admits a dominated splitting of the form $TM = E \oplus E^{uu}$, such that the subbundle $E^{uu}$ expands exponentially fast under the action of $Df$.  As one application of theorem A, we obtain the following criterion of stable Bernoulli for weakly partially hyperbolic systems.

\begin{theoremb}
\label{theoremb}
Let $f\in \mathrm{Diff}^2_{\omega}(M)$. Suppose that $f$ is weakly partially hyperbolic with dominated splitting $TM =E \oplus E^{uu}$ and chain-hyperbolic with respect to the same splitting. If $f$ has all Lyapunov exponents negative along the direction $E$ on a set of positive $m$-measure, then $f$ is stably ergodic, in fact stably Bernoulli. 
\end{theoremb}

This theorem can be seen as a version of theorem $4$ in \cite{bdp02} for weakly partially hyperbolic diffeomorphisms. We also remark that if $f\in \mathrm{Diff}^2_{\omega}(M)$ verifies the hypothesis of theorem $A$ and the direction $F$ is hyperbolic, meaning $F= E^{uu}$, then (\ref{hyp}) implies that $f$ verifies the hypothesis of theorem $B$. However, a diffeomorphism which verifies the hypothesis of theorem $B$, does not necessarily verify the hypothesis of theorem $A$, a priori.

Theorem A gives more flexibility in the construction of the example considered by Tahzibi. To construct the example one makes a deformation supported in a finite number of small balls around hyperbolic fixed points, in particular, the deformations are local. Theorem A applies to this example and quantifies, in a certain way, how much one can make such a deformation, in particular, the deformations do not have to be local. In section \ref{example} we will explain the construction of such an example in this non local way. We remark that our proof is different from the proof of Tahzibi in \cite{tahzibi04}. 

Pugh and Shub conjectured in \cite{pughshub97}, that stable ergodicity is $C^r$-dense among the partially hyperbolic conservative $C^r$-diffeomorphisms. A remarkable result by Avila, Crovisier and Wilkinson states that stable ergodicity is $C^1$-dense among the partially hyperbolic conservative $C^r$-diffeomorphisms, indeed they obtain stable Bernoulli, see theorem A' in \cite{acwergodicity}.

As another application of theorem A and some others results, we can prove the $C^1$-density of stably Bernoulli diffeomorphisms among a certain class of weakly partially hyperbolic diffeomorphisms. Let us precise this class.

Let $\mathcal{D}\subset \mathrm{Diff}^2_{\omega}(M)$ be the subset of diffeomorphisms $f$ that verifies the following properties:
\begin{itemize}
\item $f$ is weakly partially hyperbolic, with dominated splitting $TM = E\oplus E^{uu}$ and $dim(E) =2$;
\item $f$ is chain-hyperbolic for the splitting $TM = E \oplus E^{uu}$.
\end{itemize}

Define $\mathcal{WCH}^2_{\omega}(M)$ to be the $C^1$-interior of $\mathcal{D}$ for the relative topology. For the $d$-torus this set is non empty, with $d\geq 3$. The examples of Bonatti-Viana, see section $6.2$ of \cite{bvcontracting}, belong to this set. Indeed, the arguments in section \ref{example} also apply to such examples, justifying that they belong to $\mathcal{WCH}^2_{\omega}(\mathbb{T}^d)$. We have the following theorem. 

\begin{theoremc}
\label{theoremc}
Stable Bernoulli is $C^1$-dense on $\mathcal{WCH}^2_{\omega}(M)$.
\end{theoremc}

We remark that all our results remain true for $C^{1+ \alpha}$-diffeomorphisms. We conclude this introduction with some questions. 
\begin{question1}
What others criteria for stable ergodicity, or stable Bernoulli, can one obtain using chain-hyperbolicity?
\end{question1}

The example considered by Tahzibi in \cite{tahzibi04} is isotopic to a linear Anosov diffeomorphism. This particular example was first considered by Bonatti and Viana in \cite{bvcontracting} where they proved robust transitivity. This type of construction allow us to obtain diffeomorphisms that verify the hypothesis of theorem A and are not partially hyperbolic. We conclude the introduction with the following question.

\begin{question2}
Is there a diffeomorphism that verifies the hypothesis of theorem A, or theorem B, which is not isotopic to an Anosov diffeomorphism?
\end{question2}

Rafael Potrie obtains a positive answer for question 2 under some assumptions, see \cite{potrie15}.

\subsection*{Acknowledgements}
The author would like to thank Sylvain Crovisier for all his patience and guidance with this project. The author also thanks Alexander Arbieto for useful conversations on several points of this paper. The author would like to thank Welington Cordeiro, Todd Fisher, Mauricio Poletti, Rafael Potrie and Ali Tahzibi for useful comments on the paper.

\section{Preliminaries}
\label{preliminaries}
\subsection{Dominated splittings and chain hyperbolic homoclinic classes}
Let $f\in \mathrm{Diff}^1(M)$ be a $C^1$-diffeomorphism of $M$ that admits a dominated splitting $TM = E \oplus F$. It is well known that dominated splitting is a $C^1$-open property, meaning that if $g$ is sufficiently $C^1$-close to $f$, then $g$ admits a dominated splitting $T_M = E_g \oplus F_g$, where $\mathrm{dim}(E_g)= \mathrm{dim}(E)$ and $\mathrm{dim}(F_g) = \mathrm{dim}(F)$. It is also well known that the maps $g \mapsto E_g$ and $g\mapsto F_g$ are continuous in a $C^1$-neighborhood of $f$. We call $E_g$ and $F_g$ the continuations of the subbbundles $E$ and $F$. In particular, this implies that the maps
\[
\displaystyle g \mapsto \int_M \log \|Dg|_{E_g}\|dm \textrm{ and } g \mapsto \int_M \log \|Dg^{-1}|_{F_g}\|dm
\]
are continuous in a $C^1$-neighborhood of $f$.

\begin{definition}[Plaque Family]
\label{plaquefamily}
A \textit{plaque family} tangent to $E$ is a continuous map $\mathcal{W}$ from $E$ into $M$ that verifies:
\begin{itemize}
\item For each $x\in M$, the induced map $\mathcal{W}_x: E(x) \to M$ is a $C^1$-embedding which satisfies $\mathcal{W}_x(0) = x$ and whose image is tangent to $E(x)$ at $x$;
\item $(\mathcal{W}_x)_{x\in M}$ is a continuous family of $C^1$-embeddings.
\end{itemize}
A plaque family is \textit{locally invariant} if there exists $\rho>0$ such that for each $x\in M$, the image $f(\mathcal{W}_x(B(0,\rho)))$ is contained in the plaque $\mathcal{W}_{f(x)}$.
\end{definition}
The condition of dominated splitting alone cannot guarantee that the subbundles $E$ or $F$ are integrable. In \cite{hps}, the authors proved that for diffeomorphisms with dominated splitting $TM = E \oplus F$, there are always locally invariant plaque families for $f$ tangent to the direction $E$. Similarly, there are always locally invariant plaque families for $f^{-1}$ tangent to the direction $F$.  
\begin{definition}[Trapped plaques]
\label{trappedplaques}
A plaque family $(\mathcal{W}_x)_{x\in M}$ is \textit{trapped} for $f$ if for each $x\in M$, it holds
\[
f(\overline{\mathcal{W}_x}) \subset \mathcal{W}_{f(x)},
\]
where $\overline{\mathcal{W}_x}$ denotes the closure of $\mathcal{W}_x$.
\end{definition}

Let $\mathrm{Per}(f)$ be the set of all periodic points of $f$. For $p\in \mathrm{Per}(f)$, we write $O_f(p)$ the orbit of $p$ for $f$. If it is clear that we are considering the orbit for $f$ we will just write $O(p)$. A periodic point $p$ is hyperbolic if there is a dominated splitting over $O(p)$, $T_{O(p)}M = E^{ss} \oplus E^{uu}$, such that $E^{ss}$ is contracted and $E^{uu}$ is expanded exponentially fast under the action of $Df$. It is well known that if $f$ is a $C^r$-diffeomorphism, for any hyperbolic periodic point $p$ of $f$, there is an immersed $C^r$-submanifold $W^{ss}(p,f)$, called the stable manifold of $p$, which is tangent to $E^{ss}(p)$ at $p$. Similarly, there is an immersed $C^r$-submanifold called the unstable manifold, which we will denote it by $W^{uu}(p,f)$.

Let $\pi(p)\in\N$ be the period of the periodic point $p$ and write 
\[
W^{ss}(O(p),f) = \displaystyle \bigcup_{j=0}^{\pi(p)-1} W^{ss}(f^j(p),f),
\]
to be the \textit{stable manifold} of the orbit of $p$. Analogously, we define $W^{uu}(O(p),f)$.

Given two immersed submanifolds $S_1$ and $S_2$ of $M$, we say that a point $x$ is a point of {\it transverse intersection} between $S_1$ and $S_2$ if $x\in S_1 \cap S_2$ and $T_xM = T_xS_1 + T_xS_2$. We denote the set of points of transverse intersection between $S_1$ and $S_2$ by $S_1 \pitchfork S_2$.  

Given two hyperbolic periodic points $p,q\in \mathrm{Per}(f)$, we say that they are \textit{homoclinically related} if $W^{ss}(O(p),f) \pitchfork W^{uu}(O(q),f) \neq \emptyset$ and $W^{uu}(O(p),f) \pitchfork W^{ss}(O(q),f) \neq \emptyset$. We write $p \sim q$ if $p$ is homoclinically related to $q$.  

 The \textit{homoclinic class} of $p$ is defined as
\[
H(p,f) = \overline{\{q\in \mathrm{Per}(f) :p\sim q\}}.
\]
If it is clear that we are referring to $f$ we will just write $H(p)$ as the homoclinic class of $p$ for $f$. 
\begin{definition}[Chain-hyperbolicity]
\label{chainhyperbolicity}
We say that $f\in \mathrm{Diff}^r(M)$ is \textit{chain-hyperbolic} if:
\begin{enumerate}
\item there exists a periodic hyperbolic point $p$ such that $H(p) = M$;
\item there is a dominated splitting $TM = E \oplus F$;
\item there is a plaque family $(\mathcal{W}^E_x)_{x\in M}$ tangent to $E$ which is trapped by $f$. There is also another plaque family $(\mathcal{W}^F_x)_{x\in M}$ tangent to $F$ which is trapped by $f^{-1}$;
\item there are two periodic hyperbolic points, $q_s$ and $q_u$, homoclinically related to $p$ such that the stable manifold of $q_s$ contains the plaque $\mathcal{W}^E_{q_s}$ and the unstable manifold of $q_u$ contains the plaque $\mathcal{W}^F_{q_u}$.   
\end{enumerate}
\end{definition}

Let $f$ be a weakly partially hyperbolic diffeomorphism, with dominated splitting $TM = E \oplus E^{uu}$. It is well known that the subbundle $E^{uu}$ is uniquely integrable, that is, there exists an unique foliation $\mathcal{W}^{uu}$, called unstable foliation, that is tangent to $E^{uu}$. If $f$ is also chain-hyperbolic with respect to the same dominated splitting, then the plaque family for the direction $E^{uu}$ can be taken as the unstable foliation $\mathcal{W}^{uu}$. 

A key consequence of chain-hyperbolicity for us is given in the following lemma.
 
\begin{lemma}[\cite{crovisierpujalsessential}, Lemma $3.2$]
\label{keylemma}
If $f\in \mathrm{Diff}^1(M)$ is chain-hyperbolic, there exists a dense set $\mathcal{P}\subset M$ of hyperbolic periodic points homoclinically related to $p$, such that for any point $q\in \mathcal{P}$, the plaques $\mathcal{W}^E_q$ and $\mathcal{W}^F_q$ are respectively contained in the stable and in the unstable manifolds of $q$.
\end{lemma}

Given $0<\theta\leq 1$, we define the cone of size $\theta$ around the direction $E$ as
\[
\mathcal{C}^E_{\theta}=\{ (v_E, v_F) \in E \oplus F: \theta \|v_E\| \geq \| v_F\|\}.
\]

\begin{remark}
\label{chainremark}
Since both plaque families are continuous, by compactness, there exists $r>0$ such that for every $x\in M$ the plaque $\mathcal{W}_x^E$ contains a $C^1$-disc of radius $r$, centered in $x$ and tangent to $E(x)$. Furthermore, by domination, for some small $\theta>0$, we can assume that these discs are tangent to $\mathcal{C}^E_{\theta}$. An analogous result holds for the plaque family $\{\mathcal{W}^F_x\}_{x\in M}$. Thus, lemma \ref{keylemma} states that densely there are periodic points homoclinically related to $p$ whose stable and unstable manifolds have size bounded from below by $r$ and ``good" geometry, meaning controlled angles.
\end{remark}

If $f$ is chain-hyperbolic, it is easy to see that for every $N\in \N$, properties $2$ through $4$ in definition \ref{chainhyperbolicity} remain valid for $f^N$. On the other hand, it is not so immediate that property $1$ holds for $f^N$. It could happen that the entire manifold is no longer a homoclinic class and it could be divided into finitely many distinct homoclinic classes. As a consequence of lemma \ref{keylemma}, we obtain that this is not the case.

\begin{lemma}
\label{lemma.topmixing}
Let $M$ be connected. If $f$ is chain-hyperbolic, then for every $N\in \N$ it holds that $f^N$ is chain-hyperbolic.  
\end{lemma}
\begin{proof}
Let $p$ be the hyperbolic periodic point for $f$ in the definition of chain-hyperbolicity and fix $N\in \N$. It suffices to prove that $H(p,f^N) = M$. 

Let $\mathcal{P}\subset M$ be the set of hyperbolic periodic points given by lemma \ref{keylemma} for $f$. Notice that for $f^N$, the set $\mathcal{P}$ is also formed by hyperbolic periodic points with stable and unstable manifolds of uniform size, given by the plaques $\mathcal{W}^E$ and $\mathcal{W}^F$. 

Let $\varepsilon>0$ be small enough such that any two points in $q', q'' \in \mathcal{P}$ that are $\varepsilon$-close to each other verify 
\[
W^E_{q'} \pitchfork W^F_{q''} \neq \emptyset \textrm{ and } W^E_{q''} \pitchfork W^F_{q'} \neq \emptyset.
\]
In particular $q'$ and $q''$ are homoclinically related for $f^n$, for any $n\in \N$. The existence of $\varepsilon$ is a consequence of remark \ref{chainremark}. 

For any two points $q', q'' \in \mathcal{P}$ we can take a finite set of points $\{q_0, \cdots, q_k\} \subset \mathcal{P}$, such that $q_0 = q'$, $q_k = q''$ and for every $i=0, \cdots, k-1$ it holds $d(q_i, q_{i+1})< \varepsilon$. This implies that any two points $q', q''\in \mathcal{P}$ are homoclinically related for $f^N$. By the density of $\mathcal{P}$ there exists a point $q\in \mathcal{P}$ such that $q$ is homoclinically related with $p$ for $f^N$. We conclude that $H(p,f^N)=M$, which finishes the proof.
\end{proof}

\begin{remark}
\label{remark.homorel}
In the setting of lemma \ref{lemma.topmixing}, from its proof and using the inclination lemma, see lemma $7.1$ in \cite{palisdemelo}, we obtain the following: for any $\varepsilon>0$ small enough, there exists $\{q_0, \cdots, q_k\} \subset \mathcal{P}$ which is $\varepsilon$-dense, such that $W^{ss}(q_i,f) \pitchfork W^{uu}(p,f) \neq \emptyset$ and $W^{uu}(q_i,f) \pitchfork W^{ss}(p,f) \neq \emptyset$ for $i=0, \cdots, k$. We remark that here we consider the stable and unstable manifold of the point and not of the orbit. This property holds in a $C^1$-neighborhood of $f$.
\end{remark}

One defines a \textit{chain-hyperbolic homoclinic class} as a homoclinic class $H(p)$ that verifies conditions $2$ through $4$ in definition \ref{chainhyperbolicity}. We remark that the same argument as in the proof of lemma \ref{lemma.topmixing} implies that if $H(p)$ is a connected, chain-hyperbolic homoclinic class for $f$, then for every $N\in \N$ it holds that $H(p)$ is a chain-hyperbolic homoclinic class for $f^N$.

\subsection{Pesin's theory and criterion of ergodicity}
Let $f\in \mathrm{Diff}^1(M)$ be a $C^1$-diffeomorphism. A number $\lambda \in \R$ is a Lyapunov exponent of $f$ at $x$ if there exists a nonzero vector $v\in T_xM$ such that
\[
\displaystyle \lim_{n\to +\infty} \frac{1}{n}\log \|Df^n(x)v\| = \lambda.
\]
For a point $x$ and a vector $v\in T_xM$, write
\[
\displaystyle \lambda(x,v) :=\lim_{n\to +\infty} \frac{1}{n}\log \|Df^n(x)v\|.
\]
A key theorem in smooth ergodic theory is the Oseledet's theorem.
\begin{theorem}[\cite{barreirapesinbook}, Theorems 2.1.1 and 2.1.2]
\label{oseledets}
Let $f\in \mathrm{Diff}^1_{\omega}(M)$. There exists a set $\mathcal{R}_f$ of full Lebesgue measure, such that for any $x\in \mathcal{R}_f$ there is a number $1 \leq l(x) \leq  \mathrm{dim}(M)$ and there are $l(x)$ Lyapunov exponents $\lambda_1(x) < \dots < \lambda_{l(x)}(x)$. For this point $x\in \mathcal{R}_f$, there is a decomposition of the tangent space over $x$
\[
T_xM = E_1(x) \oplus \cdots \oplus E_{l(x)}(x),
\]
which is $Df$-invariant. This decomposition varies measurably with $x\in \mathcal{R}_f$ and for every $v_i\in E_i(x) -\{0\}$, it holds that $\lambda(x,v_i) = \lambda_i(x)$, for $i=1, \cdots ,l(x)$. 
\end{theorem}
A point of the set $\mathcal{R}_f$, given by the previous theorem, is called a \textit{regular point}. A $f$-invariant measure $\mu$ is \textit{non-uniformly hyperbolic} if for $\mu$-almost every point all its Lyapunov exponents are non-zero. 

For a regular point $x\in \mathcal{R}_f$, we write 
\begin{equation}
\label{oseledecsdirection}
E^s(x) = \displaystyle \bigoplus_{i: \lambda_i(p)<0} E_i(p) \textrm{ and } E^u(p) = \bigoplus_{i: \lambda_i(p)>0} E_i(p).
\end{equation}

\begin{definition}
\label{pesinmanifolddef}
For a $C^2$-diffeomorphism $f\in \mathrm{Diff}^2_{\omega}(M)$, the \textit{stable Pesin manifold} of the point $x\in \mathcal{R}_f$ is
\[
W^s(x,f) =\{ y\in M: \displaystyle \limsup_{n\to +\infty} \frac{1}{n} \log d(f^n(x), f^n(y)) <0 \}.
\]
Similarly one defines the \textit{unstable Pesin manifold} as
\[
W^u(x,f) = \{ y\in M: \displaystyle \limsup_{n\to +\infty} \frac{1}{n} \log d(f^{-n}(x), f^{-n}(y)) <0 \}.
\]
\end{definition}

Let $f\in \mathrm{Diff}^2_{\omega}(M)$, for Lebesgue almost every point $x\in \mathcal{R}_f$, the Pesin's manifolds are immersed $C^1$-submanifolds, see section $4$ of \cite{pesin77}. Let $p\in \mathrm{Per}(f)$ be a hyperbolic periodic point. Define the following sets:
\[
\begin{array}{rcl}
H^s(O(p)) & = & \{x\in \mathcal{R}_f: W^s(x,f) \pitchfork W^{uu}(O(p),f) \neq \emptyset\},\\
H^u(O(p)) & = & \{x\in \mathcal{R}_f: W^u(x,f) \pitchfork W^{ss}(O(p),f) \neq \emptyset\}.
\end{array}
\]
Define the \textit{ergodic homoclinic class} of $p$ by 
\[
H_{erg}(O(p)) = H^s(O(p)) \cap H^u(O(p)).
\]
It is easy to see that $H_{erg}(p)$ is $f$-invariant. Given two measurable sets $A , B \subset M$ we write $A \circeq B$, if $A$ only differs from $B$ in a set of zero Lebesgue-measure. Given a measurable set $\Lambda$ with positive $m$-measure, we define $m_{\Lambda}$ to be the normalized restriction of the measure $m$ to the set $\Lambda$, that is, for any measurable set $A$, 
\[
m_{\Lambda}(A) = \frac{m(A\cap \Lambda)}{m(\Lambda)}.
\]
The following theorem will give us a criterion for ergodicity.
\begin{theorem}[\cite{hhtu11}, Theorem A]
\label{hhtucriteria}
Let $f\in \mathrm{Diff}^2_{\omega}(M)$. For a hyperbolic periodic point $p\in Per(f)$, if $m(H^s(O(p)))>0$ and $m(H^u(O(p)))>0$, then 
\[
H_{erg}(O(p)) \circeq H^s(O(p)) \circeq H^u(O(p)).
\]
Moreover $f|_{H_{erg}(O(p))}$ is ergodic and non-uniformly hyperbolic, with respect to the measure $m_{H_{erg}(O(p))}$. In particular, if $m(H_{erg}(O(p))) =1$ then f is ergodic.
\end{theorem}

We will also need the following result by Pesin.

\begin{theorem}[\cite{pesin77}, Theorem 8.1]
\label{theorem.pesinbernoulli}
Let $f$ be a $C^2$-diffeomorphism preserving a smooth measure $m$. Suppose that $f$ is non-uniformly hyperbolic and ergodic for the measure $m$. Then there exist $K\in \N$ and measurable sets with positive $m$-measure $\Lambda_1, \cdots \Lambda_K$ which are pairwise disjoints, such that $f(\Lambda_i) = \Lambda_{i+1}$ for $i=1, \cdots, K-1$, $f(\Lambda_K ) =  \Lambda_1$ and for each $j=1, \cdots, K$, the system $(f^K, m|_{\Lambda_j})$ is Bernoulli. In particular, if $K=1$ then $(f,m)$ is Bernoulli. 
\end{theorem}

For a hyperbolic periodic point $p$, define $h^s(p) =  \{x\in \mathcal{R}_f: W^s(x,f) \pitchfork W^{uu}(p,f) \neq \emptyset\}$, notice that in this definition we are taking the unstable manifold of the point $p$ and not the unstable manifold of the orbit of $p$. Analogously, we define the set $h^u(p)$. We define the \textit{pointwise ergodic homoclinic class} as 
\begin{equation}
\label{eq.pointwisehom}
h_{ber}(p)= h^s(p) \cap h^u(p).
\end{equation}
As a consequence of theorems \ref{hhtucriteria} and \ref{theorem.pesinbernoulli}, we obtain the following corollary.
\begin{corollary}
\label{cor.bernoulli}
Let $f\in \mathrm{Diff}^2_{\omega}(M)$. For a hyperbolic periodic point $p\in Per(f)$, with period $\pi(p)$, if $m(h^s(p))>0$ and $m(h^u(p))>0$, then 
\[
h_{ber}(p) \circeq h^s(p) \circeq h^u(p).
\]
Moreover $f^{\pi(p)}|_{h_{ber}(p)}$ is Bernoulli and non-uniformly hyperbolic, with respect to the measure $m_{h_{ber}(p)}$. In particular, if $m(h_{ber}(p)) =1$ then $(f,m)$ is Bernoulli.

\end{corollary} 
\begin{proof}
Apply theorem \ref{hhtucriteria} for $f^{\pi(p)}$ and conclude that $h_{ber}(p) \circeq h^s(p) \circeq h^u(p)$. Applying theorem \ref{hhtucriteria} for $f$, we obtain that $H_{erg}(O(p))$ is a non-uniformly hyperbolic ergodic component of $f$. Using theorem \ref{theorem.pesinbernoulli} and the fact that $f^{\pi(p)}(h_{ber}(p)) = h_{ber}(p)$ we conclude that $(f^{\pi(p)}|_{h_{ber}(p)},m|_{h_{ber}(p)})$ is Bernoulli. Again by theorem \ref{theorem.pesinbernoulli}, if $m(h_{ber}(p)) =1$ then $(f,m)$ is Bernoulli.  
\end{proof}

\section{Proof of theorem A}

Fix $N\in \N$, such that for any $x\in M$ it holds
\[
\|Df^N|_{E(x)}\|. \|Df^{-N}|_{F(f^N(x))}\| < \frac{1}{2}.
\]
Let $\mathcal{U}_1$ be a $C^1$-neighborhood of $f$ such that for any $g\in \mathcal{U}_1$ it is verified 
\[
\|Dg^N|_{E_g(x)}\|. \|Dg^{-N}|_{F_g(g^N(x))}\| < \frac{1}{2},
\]
where $E_g(.)$ and $F_g(.)$ are the continuations of the subbundles $E$ and $F$, which we defined at the  beginning of section \ref{preliminaries}. For each $g\in \mathcal{U}_1$, define the auxiliary functions
\begin{equation}
\label{funcoesauxiliares}
\varphi_g(x) = \log \|Dg^N|_{E_g(x)}\| \textrm{ and } \psi_g(x) = \log \|Dg^{-N}|_{F_g(x)}\|.
\end{equation}

By our assumption (\ref{hyp}), we can take $\beta >0$, a constant, such that  
\[
\displaystyle \int_M \log \|Df|_{E}\| dm <-2\beta \textrm{ and } \int_M \log \|Df^{-1}|_{F}\| dm <-2\beta.
\]
By our discussion at the beginning of section \ref{preliminaries}, we can assume that $\mathcal{U}_1$ is small enough such that for any $g\in \mathcal{U}_1$, it is verified that
\[
\displaystyle \int_M \log \|Dg|_{E_g}\| dm < -\beta \textrm{ and } \int_M \log \|Dg^{-1}|_{F_g}\| dm < -\beta.
\]
Let $\sigma = \min \{ \frac{\log 2}{2}, \beta\}$ and observe that for every $g\in \mathcal{U}_1$, it holds that 
\begin{equation}
\label{hipforN}
\displaystyle \int_M \varphi_g dm < -\sigma \textrm{ and } \int_M \psi_g dm < -\sigma.
\end{equation}
For each $g\in \mathcal{U}_1$, define the sets
\[\arraycolsep=1.2pt\def\arraystretch{3}
\begin{array}{rcl}
A_g& = & \left\{ x\in M : \displaystyle \limsup_{n\to +\infty} \frac{1}{n}\sum_{j=0}^{n-1} \log \|Dg^N|_{E_g(g^{jN}(x))}\| \leq -\sigma\right\};\\
B_g & = & \left\{ x\in M : \displaystyle \limsup_{n\to +\infty} \frac{1}{n}\sum_{j=0}^{n-1} \log \|Dg^{-N}|_{F_g(g^{-jN}(x))}\| \leq -\sigma\right\}.
\end{array}
\]
We have the following lemma.
\begin{lemma}
\label{lemma2}
For every $g\in \mathcal{U}_1\cap \mathrm{Diff}^2_{\omega}(M)$, both $A_g$ and $B_g$ have positive $m$-measure.
\end{lemma}
\begin{proof}
Let us prove that $A_g$ has positive measure, the proof is analogous for $B_g$. From (\ref{hipforN}), we have 
\[
\displaystyle \int_M \varphi_g dm < -\sigma.
\] 
Consider the Birkhoff average
\[
\displaystyle \widetilde{\varphi}_g(.) := \lim_{n \to + \infty} \frac{1}{n} \sum_{j=0}^{n-1} \varphi_g \circ g^{jN}(.).
\]
By Birkhoff's theorem, $\widetilde{\varphi}_g$ is defined for almost every point and 
\begin{equation}
\label{inequalityforlemma}
\displaystyle \int_M \widetilde{\varphi}_g dm = \int_M \varphi_g dm < -\sigma.
\end{equation}
Observe that the set $A=\{ x\in M: \widetilde{\varphi}_g(x) < -\sigma\}$ is contained in $A_g$. From (\ref{inequalityforlemma}), we conclude that $A$ has positive measure, which implies that $A_g$ has positive measure as well.
\end{proof}
This lemma will allow us to verify the conditions for theorem \ref{hhtucriteria} to hold. Using the domination we can prove the following lemma.
\begin{lemma}
\label{lemma1}
For every $g\in \mathcal{U}_1 \cap \dworld$ it holds that $m(A_g \cup B_g) = 1$.
\end{lemma}
\begin{proof}
Let $g\in \mathcal{U}_1$ and $\mu$ be a $g^N$-invariant ergodic measure. Suppose that $\mu(A_g) = 0$. The domination implies that for every $x\in M$
\[
\varphi_g(x) + \psi_g \circ g^N(x) < -\log 2.
\]

Since $\mu$ is ergodic, for $\mu$-almost every point $x\in M- A_g$ it holds that 
\[
\int_M \varphi_g d\mu = \lim_{n\to + \infty} \frac{1}{n} \sum_{j=0}^{n-1} \varphi_g \circ g^{jN}(x) >-\sigma \geq -\frac{\log 2}{2}.
\]

Thus, by domination

\[
\int_M \psi_g\circ g^N d\mu = \int_M \psi_g d\mu < - \log 2 + \frac{\log 2}{2} = - \frac{\log 2}{2} \leq -\sigma. 
\]
Since $\mu$ is ergodic, for $\mu$-almost every point $x$, it holds
\[
\displaystyle \lim_{n\to +\infty} \frac{1}{n} \sum_{j=0}^{n-1} \psi_g \circ g^{-jN}(x) = \int_M \psi_g d\mu < -\sigma.
\]

In particular, $\mu(B_g)=1$. Since the sets $A_g$ and $B_g$ are invariant, we obtain that for any ergodic measure $\mu$ it holds that $\mu(A_g \cup B_g) =1$. Using the ergodic decomposition theorem, see theorem $6.4$ in \cite{manebook}, we conclude that $m(A_g \cup B_g) =1$.

\end{proof}

For $g\in \mathcal{U}_1\cap \mathrm{Diff}^2_{\omega}(M)$, recall that $\mathcal{R}_g$ is the set of regular points for $g$. For a regular point $x\in A_g\cap \mathcal{R}_g$ all the Lyapunov exponents for $g^N$ on $E_g(x)$ are negative. Indeed,
\[
\displaystyle \lim_{n\to +\infty} \frac{1}{n} \log \|Dg^{nN}|_{E_g(x)}\| \leq \lim_{n\to +\infty} \frac{1}{n} \sum_{j=0}^{n-1} \log \|Dg^{N}|_{E_g(g^{jN}(x))}\| < -\sigma .
\] 
For $x\in \mathcal{R}_g$, consider the stable Pesin manifold $W^s(x,g)$ for $g^N$ and for $g$. Similarly we define those sets for the unstable Pesin manifold and we denote it by $W^u(x,g)$. 

\begin{lemma}
\label{lemma3tamanhouniforme}
There are a $C^1$-neighborhood $\mathcal{U}_2 \subset \mathcal{U}_1$ of $f$ and two constants $r_0, \theta_0>0$ that verify the following: For $g\in \mathcal{U}_2\cap \mathrm{Diff}^2_{\omega}(M)$ and for any $x \in A_g\cap \mathcal{R}_g$ there exists $n\geq 0$, such that $W^s(g^{-nN}(x),g)$ contains a $C^1$-disc of radius $r_0$, centered in $g^{-nN}(x)$ and tangent to $\mathcal{C}^E_{\theta_0}$. 
\end{lemma}
The proof of the existence of $r_0>0$ can be found in lemma 2 of \cite{bdp02}. The proof uses the notion of hyperbolic times and ideas from \cite{alvesbonattiviana}. The existence of $\theta_0$ follows from domination. A similar result holds for $B_g$ and we can suppose that $r_0$ and $\theta_0$ are the same for both sets, $A_g$ and $B_g$.

\begin{remark}
In theorem $3.11$ in \cite{abdenurbonatticrovisier2011}, the authors prove the existence of Pesin manifolds for $C^1$-diffeomorphisms with a dominated splitting. From this result, we conclude that the conclusion of lemma \ref{lemma3tamanhouniforme} also holds for $g\in \mathcal{U}_2 \cap \mathrm{Diff}^1_{\omega}(M)$. 
\end{remark}

We remark that in the proof of lemmas \ref{lemma2}, \ref{lemma1} and \ref{lemma3tamanhouniforme}, we do not use the chain-hyperbolicity condition. These lemmas are true for any $C^1$-diffeomorphism that preserves volume and that admits a dominated splitting $TM=E \oplus F$ which verifies (\ref{hyp}).

By hypothesis $f$ is chain-hyperbolic. Let $p\in Per(f)$ be the hyperbolic point in the definition of chain-hyperbolicity such that $H(p) = M$, see definition \ref{chainhyperbolicity}. We may assume that $ \mathcal{U}_2$ is small enough such that for any $g\in \mathcal{U}_3$, there exists a hyperbolic periodic point $p_g\in Per(g)$, which is the continuation of the periodic point $p$.

Let $\mathcal{P}$ be the dense set of hyperbolic periodic points given by lemma \ref{keylemma}. By remark \ref{chainremark}, there exist two constants $r_1, \theta_1>0$ and a dense set of hyperbolic periodic points $\mathcal{P}$ homoclinically related with $p$, such that for any $q\in \mathcal{P}$, the stable manifold of $q$ contains a $C^1$-disc centered in $q$, with radius $r_1$ and tangent to $\mathcal{C}^E_{\theta_1}$ and the unstable manifold of $q$ contains a $C^1$-disc centered in $q$ of radius $r_1$ and tangent to $\mathcal{C}_{\theta_1}^F$. 

This is the main property that we use from chain-hyperbolicity. As a consequence of that, we obtain the following proposition.

\begin{proposition}
\label{mainprop}
There exists a $C^1$-neighborhood $\mathcal{U}_3 \subset \mathcal{U}_2$ of $f$, such that for any $g\in \mathcal{U}_3 \cap \mathrm{Diff}^2_{\omega}(M)$ and for any $x\in A_g \cap \mathcal{R}_g$, it is verified
\[
W^s(x,g) \pitchfork W^{uu}(p_g,g)\neq \emptyset.
\] 
Similarly, for any $y\in B_g$, it holds that $W^u(y,g) \pitchfork W^{ss}(p_g,g) \neq \emptyset$.
\end{proposition} 
\begin{proof}
Take $r= \frac{\min\{r_0, r_1\}}{2}$ and $\theta = 2\max \{\theta_0, \theta_1\}$. It is easy to see that there is $\varepsilon>0$ such that any two points $x$ and $y$ with $d(x,y)<\varepsilon$, verify the following: any two $C^1$-discs $D_1$ and $D_2$, centered in $x$ and $y$, respectively, with radius $r$ and such that $D_1$ is tangent to $\mathcal{C}^E_{\theta}$ and $D_2$ is tangent to $\mathcal{C}^F_{\theta}$, have a transverse intersection. Fix such $\varepsilon>0$.

By remark \ref{remark.homorel}, fix a finite set of hyperbolic periodic points for $f$, $\{q_0, \cdots, q_k\} \subset \mathcal{P}$, which is $\frac{\varepsilon}{2}$-dense on $M$, such that 
\begin{equation}
\label{eq.intersection}
W^{ss}(q_i,f) \pitchfork W^{uu}(p,f) \neq \emptyset \textrm{ and } W^{uu}(q_i,f) \pitchfork W^{ss}(p,f) \neq \emptyset.
\end{equation}
Consider a $C^1$-neighborhood $\mathcal{U}_3 \subset \mathcal{U}_2$ of $f$, small enough, such that for any $g\in \mathcal{U}_3\cap \mathrm{Diff}^2_{\omega}(M)$ the following properties are verified:
\begin{itemize}
\item For any $i=0 \cdots, k$, the continuation $q_{i,g}$ is defined and (\ref{eq.intersection}) holds for $q_{i,g}$ and $p_g$;
\item the set $\{q_{0,g}, \cdots, q_{k,g}\}$ is $\varepsilon$-dense on $M$;
\item the stable manifold of $q_{i,g}$ contains a $C^1$-disc centered in $q_{i,g}$, of radius $r$ and tangent to $\mathcal{C}^E_{\theta}$, for every $i=1, \cdots, m$. Similarly, the unstable manifold of $q_{i,g}$ contains a $C^1$-disc centered in $q_{i,g}$, of radius $r$ and tangent to $\mathcal{C}^F_{\theta}$.
\end{itemize}

Let $x\in A_g \cap \mathcal{R}_g$. By lemma \ref{lemma3tamanhouniforme} and by our choice of $r$ and $\theta$, there exists some $n\geq 0$ such that $W^s(g^{-nN}(x))$ contains a $C^1$-disc of radius $r$ and tangent to $\mathcal{C}^E_{\theta}$. There is some hyperbolic periodic point $q_{i,g}$ which is $\varepsilon$-close to $g^{-nN}(x)$, thus 
\[
W^s(g^{-nN}(x),g) \pitchfork W^{uu}(q_{i,g},g) \neq \emptyset,
\]
By (\ref{eq.intersection}) for $q_{i,g}$ and $p_g$ and by the inclination lemma, see lemma $7.1$ in \cite{palisdemelo}, we conclude that $W^s(x,g) \pitchfork W^{uu}(p_g,g) \neq \emptyset$. The argument is analogous for $x\in B_g \cap \mathcal{R}_g$.   
\end{proof}

We remark that a homoclinic class with dominated splitting has a dense set of periodic points such that each of these points has an iterate with either the stable or unstable manifold of uniform size. Without the chain-hyperbolicity condition, we cannot guarantee the existence of a dense set of periodic points whose both stable and unstable manifolds have uniform size, this property was crucial in the proof of proposition \ref{mainprop}. 

Let us prove that any $g\in \mathcal{U}_3 \cap \textrm{Diff}^2_{\omega}(M)$ is Bernoulli. Recall that we defined in section \ref{preliminaries} the sets $h^s(p_g)$ and $h^u(p_g)$. By proposition \ref{mainprop}, we have $\left(A_g\cap \mathcal{R}_g\right) \subset h^s(p_g)$ and $\left(B_g \cap \mathcal{R}_g \right) \subset h^u(p_g)$. Since the set of regular points $\mathcal{R}_g$ has full measure, by lemma \ref{lemma2} we conclude that 
\[
m(h^s(p_g))>0 \textrm{ and } m(h^u(p_g))>0.
\]
Corollary \ref{cor.bernoulli} implies that $h_{ber}(p_g)\circeq h^s(p_g) \circeq h^u(p_g)$ and $(g^{\pi(p_g)}|_{h_{ber}(p_g)},m|_{h_{ber}(p)})$ is Bernoulli. By lemma \ref{lemma1}, we obtain that $m(h_{ber}(p_g))=1$, which implies that $g$ is Bernoulli.

\begin{remark}
Lemmas \ref{lemma2}, \ref{lemma1} and proposition \ref{mainprop} also hold for diffeomorphisms $g\in \mathcal{U}_i \cap \mathrm{Diff}^1_{\omega}(M)$, for $i=1,2$. 
\end{remark}

\section{Proof of theorem B}
\label{proof.theoremb}

Recall that if $f$ is a weakly partially hyperbolic diffeomorphism with dominated splitting $TM= E \oplus E^{uu}$, then the unstable direction is uniquely integrable by a foliation $\mathcal{W}^{uu}$. For a point $x\in M$, let $W^{uu}(x)$ be the leaf that contains the point $x$. 
 
\begin{definition}
\label{def.dynamicalminimal}
The unstable foliation of a weakly partially hyperbolic diffeomorphisms $f$ is {\it dynamically minimal} if for any point $x\in M$, the set 
\[
\displaystyle W^{uu}(O(x)):=\bigcup_{n\in \Z} W^{uu}(f^n(x),f)
\]
is dense on the manifold.
\end{definition}
The key property in the proof of theorem B is given in the following proposition, which is a consequence of chain-hyperbolicity.
\begin{proposition}
\label{prop.dynamicalminimality}
Let $f$ be a $C^1$-diffeomorphism, which does not have to preserve a volume form. If $f$ is weakly partially hyperbolic with dominated splitting $TM=E \oplus E^{uu}$ and is chain-hyperbolic with respect to the same splitting, then the unstable foliation $\mathcal{W}^{uu}$ is dynamically minimal. 
\end{proposition}

\begin{proof}
Let $p\in M$ be a hyperbolic periodic point such that $H(p) =M$, given in the definition of chain-hyperbolicity. By the definition of homoclinic class it is immediate that $\mathcal{W}^{uu}_f(O(p))$ is dense on $M$. By lemma \ref{keylemma}, there is a dense set $\mathcal{P}$ of periodic points homoclinically related to $p$, such that for any $q\in \mathcal{P}$ its stable manifold contains the plaque $\mathcal{W}^E_q$. In particular every $q\in \mathcal{P}$ also verifies that $W^{uu}(O(q))$ is dense on $M$. 

For any point $x\in M$, there exists a periodic point $q \in \mathcal{P}$ such that $W^{uu}(x) \pitchfork \mathcal{W}^E_q \neq \emptyset$. This is an immediate consequence of the density of the set $\mathcal{P}$, the uniform size of the plaques $\mathcal{W}^E_q$ and the fact that such plaques are tangent to a cone $\mathcal{C}^E_{\theta}$ for some small $\theta$. The proposition then follows by the inclination lemma, see lemma $7.1$ in \cite{palisdemelo}. 
\end{proof}

Let $f\in \mathrm{Diff}^2_{\omega}(M)$ be a weakly partially hyperbolic diffeomorphism. For any $x\in M$ consider two small discs $T_1$ and $T_2$ close to $x$ and transverse to $W^{uu}_{loc}(x)$. The unstable holonomy between $T_1$ and $T_2$ is the map $H:T_1 \to T_2$ defined as $H(q) = W^{uu}_{loc}(q) \cap T_2$, this map is well defined by the transversality of the discs $T_1$ and $T_2$ and the fact that $W^{uu}_{loc}(q)$ vary continuously with the choice of $q$. Since $f$ is $C^2$, it is well known that the unstable foliation is \textit{absolutely continuous}, that is, the map $H$ takes sets of zero Lebesgue measure inside $T_1$ into sets of zero measure inside $T_2$. 

In the $C^2$-scenario we obtain the following lemma, which is an adaptation of an argument due to Brin in \cite{brin75}, for the weakly partially hyperbolic scenario.

\begin{lemma}
\label{lemma.metrictransitivity}
Let $f\in \mathrm{Diff}^2_{\omega}(M)$ be a weakly partially hyperbolic diffeomorphism with dominated splitting $TM=E \oplus E^{uu}$. If $f$ is chain-hyperbolic with respect to the same splitting, then $m$-almost every point has dense orbit.
\end{lemma}
\begin{proof}
First observe that it is enough to prove that for any open set $U\subset M$, the set of points whose orbit intersects $U$ has full $m$-measure. A point $x\in M$ is backwards recurrent if it is an accumulation point of the sequence $(f^{-n}(x))_{n\in \N }$. Let $R\subset M$ the set of backwards recurrent points, by Poincar\'e recurrence theorem this set has full $m$-measure. It is a classical consequence of the absolute continuity of the unstable foliation that there exists a set $\Lambda \subset M$ of full $m$-measure such that for any point $x\in \Lambda$, the set $W^{uu}(x) \cap R$ has full Lebesgue measure inside the submanifold $W^{uu}(x)$, see for instance lemma $5$ in \cite{bdp02}. Observe that we can suppose that the same holds for $f^n(x)$, for any $n\in \Z$.

Fix an open set $U$ and take $x\in \Lambda$. By proposition \ref{prop.dynamicalminimality}, there exists $k\in \Z$ such that $W^{uu}(f^k(x)) \cap U \neq \emptyset$. Since the set $U$ is open, the set $W^{uu}(f^k(x)) \cap U $ has positive Lebesgue measure inside $W^{uu}(f^{k}(x))$. In particular, there exits a point $y\in W^{uu}(f^k(x)) \cap U$ which is backwards recurrent. Since the unstable manifold contracts for backwards iterates and by the backwards recurrence of $y$, there exists $n \in \N$ such that $f^{k-n}(x) \in U$. This concludes the proof of the lemma.
\end{proof}

We now proceed to the proof of theorem B. Let $f$ be a diffeomorphism verifying the hypothesis of theorem B. Let $A\subset M$ be the $f$-invariant set of points such that all the Lyapunov exponents along the direction $E$ are negative. By hypothesis, the set $A$ has positive measure. Since the direction $E^{uu}$ is uniformly hyperbolic, a standard argument using Birkhoff's ergodic theorem and the absolute continuity of the Pesin manifolds and the strong unstable foliation, implies that every ergodic component of $f|_A$ coincides with an open set (mod 0), see for instance the proof of theorem $1$ in \cite{bdp02}. By lemma \ref{lemma.metrictransitivity}, $m$-almost every point has dense orbit. We can easily conclude that $f$ is ergodic.

By ergodicity, $m$-almost every point has all its exponents negative along the direction $E$. This implies that for $N\in \N$ large enough
\[
\displaystyle \int \log \|Df^N|_E\| dm <0.
\]
By lemma \ref{lemma.topmixing}, the diffeomorphism $f^N$ is chain-hyperbolic. Theorem A implies that $f^N$ is stably Bernoulli. Let $p\in Per(f)$ be the hyperbolic periodic point in the definition of chain hyperbolicity. From the proof of theorem A, for any $g\in \mathrm{Diff}^2_{\omega}(M)$ in a $C^1$-neighborhood of $f$, it holds that $m(h_{ber}(g))=1$, which implies that $f$ is stably Bernoulli.

\section{Proof of theorem C}
\label{prooftheoremc}

Let $f\in \mathcal{WCH}^2_{\omega}(M)$. There exists a dominated splitting $TM = E \oplus E^{uu}$ such that $dim(E)=2$. We separate the proof of theorem $C$ in two cases. The first case is when there exists a sequence $(g_n)_{n\in \N} \mathcal{WCH}^2_{\omega}(M)$, such that for each $g_n$ the subbundle $E_{g_n}$ admits a dominated splitting into two one dimensional bundles $E= E^1_{g_n} \oplus E^2_{g_n}$. The second case is when $C^1$-robustly inside $\mathcal{WCH}^2_{\omega}(M)$ the center direction does not admit any further dominated decomposition.

\textbf{Case 1:} In this case we have that arbitrarily $C^1$-close to $f$, there exists a diffeomorphism $g\in \mathcal{WCH}^2_{\omega}(M)$ such that $TM = E^1_g \oplus E^2_g \oplus E^{uu}$, with $dim(E^i_g) =1$, for $i=1,2$. Since $E^1_g$ is one dimensional and $g$ preserves volume, it follows that $E^1_g$ is uniformly contracted, see proposition $0.5$ in \cite{bdpdominated}. Hence, we have a partially hyperbolic diffeomorphisms with one dimensional center. By theorem A' in \cite{acwergodicity}, we have that $g$ is $C^1$-approximated by a stably Bernoulli diffeomorphism. 

\textbf{Case 2:} In this case, using theorem $A$ from \cite{acwexponent}, we take a diffeomorphism $g\in \mathrm{Diff}^1_{\omega}(M)$ arbitrarily $C^1$-close to $f$, which is non-uniformly hyperbolic and has negative exponent in the direction $E_g$. Thus, for $N\in \N$ large enough
\begin{equation}
\label{eq.agoravai}
\displaystyle \int_{M} \log \|Dg^N|_{E_g}\| dm < 0.
\end{equation}
Recall that condition (\ref{eq.agoravai}) is $C^1$-open. By theorem $1$ from \cite{avilaregular}, we can take a diffeomorphism $\tilde{g}\in \mathrm{Diff}^2_{\omega}(M)$ arbitrarily $C^1$-close to $g$ such that $\tilde{g}$ verifies (\ref{eq.agoravai}). By the definition of $\mathcal{WCH}^2_{\omega}(M)$, we can assume that $\tilde{g}\in \mathcal{WCH}^2_{\omega}(M)$. By lemma \ref{lemma.topmixing}, $\tilde{g}^N$ is chain-hyperbolic. Using theorem A, we conclude that $\tilde{g}^N$ is stably Bernoulli. By similar reason as in the end of the proof of theorem B, we conclude that $\tilde{g}$ is stably Bernoulli.

\section{The example}
\label{example}

In theorem C of \cite{bvcontracting}, the authors give the first example of a robustly transitive diffeomorphism having no invariant hyperbolic subbundle. Tahzibi proved in \cite{tahzibi04} the stable ergodicity of this example. The construction is made by deforming an Anosov diffeomorphism inside small balls. In this section we explain the construction of the example in a not so local way so that the hypothesis of theorem A holds. In a certain way theorem A quantifies how much the Anosov diffeomorphism can be deformed and keep the stable ergodicity.

Let $A\in SL(4,\mathbb{Z})$ be a hyperbolic matrix with four distinct eigenvalues $0<\lambda_{ss} < \lambda_s<1<\lambda_u < \lambda_{uu}$, with unit eigenvectors $e_{ss}$, $e_{s}$, $e_{u}$ and $e_{uu}$. On $\R^4$ consider the coordinate system formed by the basis $\{e_{ss}, e_s, e_u, e_{uu}\}$. We write $A^s$ the restriction of $A$ to the stable directions and $A^u$ the restriction of $A$ to the unstable directions. 

Consider the Anosov diffeomorphisms $f_A: \T^4 \to \T^4$ induced by $A$, with hyperbolic splitting $T\T^4 = E^{ss} \oplus E^s \oplus E^u \oplus E^{uu}$. Up to taking an iterate of $f_A$, we may suppose that $f_A$ has two fixed points $p_1$ and $p_2$.

For each $a,b \in (0,1)$, let $U^{a,b}_1$ and $U^{a,b}_2$ be neighborhoods of $p_1$ and $p_2$, respectively, defined as follows: let $\exp_{p_1}: T_{p_1}\T^4 \to \T^4$ be the exponential map on the point $p_1$ and define $U^{a,b}_1 = \exp_{p_1}(D^2_a \times D^2_b)$, where $D^2_a \times D^2_b$ is the product of two discs of radius $a$ and $b$, respectively, and $D^2_a$ is contained in the subspace generated by $\{e_{ss}, e_s\}$ and $D^2_b$ on the subspace generated by $\{e_u, e_{uu}\}$. Similarly we define $U_2^{a,b} = \exp_{p_2}(D^2_a \times D^2_b)$. Observe that the exponential map, $\exp_{p_i}(.)$, sends sets of the form $\{x\} \times D^2_b$ on unstable manifolds of $f_A$ inside $U^{a,b}_i$, for $i=1,2$ and $x\in D^2_a$. Similarly it sends sets of the form $D^2_a \times \{y\}$ on stable manifolds of $f_A$.

Fix $R>0$ such that for $a$ and $b$ sufficiently small, $U^{a, R}_1 \cap U^{R,b}_2 = \emptyset$. Write $U^a_1 = U^{a,R}_1$ and $U^b_2 = U^{R,b}_2$. We will describe the construction in $U^a_1$, the construction in $U^b_2$ is analogous.

Let $g:D^2_1 \times D^2_R \to D^2_1 $, be a smooth map with the following properties:
\begin{enumerate}
\item for each $y\in D^2_R$, $g(.,y)$ is a diffeomorphism of $D^2_1$, which is the identity in a neighborhood of the boundary of $D^2_1$ and preserves area;
\item $g(., y) $ is the identity if $y$ belongs to a neighborhood of the boundary of $D^2_R$;
\item $\|D_x g\| < \lambda^u$;
\item for $y=0$, the composition $g_0(A^s(x)) = g(A^s(x),0)$ has three fixed points on $D^2_1$, one saddle and two sinks, where one of the sinks has a complex eigenvalue.
\end{enumerate}

Such map can be obtained using Hamiltonian flows, see section $6$ of \cite{bvcontracting}. For each $a\in (0,1)$, consider the diffeomorphism
$$
\begin{array}{rcl}
\tilde{g}_a: D^2_a \times D^2_R & \longrightarrow & D^2_a \times D^2_R\\
(x,y) & \mapsto & (a g( a^{-1} x, y), y).
\end{array}
$$

By properties $1$ and $2$ of the map $g$, using the exponential chart $\exp_{p_1},$ we extend the diffeomorphism $\tilde{g}_a$ to a diffeomorphism $G_a$ of $\T^4$, such that $G_a(q) = \exp_{p_1} \circ \tilde{g}_a \circ \exp^{-1}_{p_1}(q)$, if $q\in U^a_1$ and $G_a(q) = q $ otherwise. By item $1$, we have that $G_a$ preserves volume. For each $a\in (0,1)$, consider the diffeomorphism $f_a = G_a \circ f_A$ of $\T^4$. Property $4$ of $g$ implies that $f_a$ has a fixed point of index $1$ and another fixed point of index $2$ with complex stable eigenvalue. 

Observe that if $f_A(x,y) \in U^a_1$, then using the coordinates $(E^{ss} \oplus E^s) \oplus ( E^u \oplus E^{uu})$ we obtain
\[
Df_a(x,y)= 
\begin{pmatrix}
D_x g(a^{-1} A^s(x), A^u(y)) A^s & a D_y g(a^{-1}A^s(x),A^u(y))A^u \\
0 & A^u
\end{pmatrix}.
\]
Property $3$ of the map $g$ implies that $\|D_xg\| .\|A^s\| < \lambda^u$. Thus, if $a$ is small enough $Df_a$ expands vectors uniformly inside a thin cone $\mathcal{C}^u$ around the directions $E^u \oplus E^{uu}$, there is a dominated splitting $T\T^4 = E^{cs}_a \oplus E^{u}_a$ such that $dim(E^{cs}_a) =2$ and $E^{cs}_a$ does not admit any further decomposition.

By a similar construction, exchanging $U^a_1$ by $U^b_2$ and exchanging the roles of the stable and unstable directions, for each $b\in (0,1)$, we obtain a volume preserving diffeomorphism $H_b$ of $\T^4$. We consider the two parameter family of conservative diffeomorphisms $f_{a,b} = H_b \circ G_a \circ f_A$. We now describe a few properties that $f_{a,b}$ has for $a$ and $b$ small.  

\begin{enumerate}[label=(\alph*)]
\item $f_{a,b}$ admits a dominated splitting of the form $T\T^4 = E^{cs}_{ab} \oplus E^{cu}_{ab}$, where $dim(E^{cs}_{ab}) =2$. It does not admit any further dominated decomposition. We also have that $E^{cs}_{ab}$ converges to $E^{ss} \oplus E^s$ and $E^{cu}_{ab}$ converges to $E^u \oplus E^{uu}$ when $a,b$ goes to zero.
\item $f_{a,b}$ has one periodic point of index $1$ and one periodic point of index $3$.
\item There is a thin cone $\mathcal{C}^s$ around the direction $E^{cs}_{ab}$, such that if $x\notin U^b_2$ then $Df_{a,b}^{-1}(x)$ expands vectors uniformly in $\mathcal{C}^s$. Similarly, there is a thin cone $\mathcal{C}^u$ around the direction $E^{cu}_{ab}$, such that if $x\notin U^a_1$ then $Df_{a,b}(x)$ expands vectors uniformly inside $\mathcal{C}^u$.
\item It holds that
\begin{equation}
\label{eq.inequalityintegral}
\displaystyle \int_{\T^4} \log\|Df_{a,b}|_{E^{cs}_{ab}}\| dm < 0 \textrm{ and } \int_{\T^4} \log \|Df^{-1}_{a,b}|_{E^{cu}_{ab}}\|dm <0.
\end{equation}   
Let us explain why this property holds. Notice that there exists a constant $C_1$ such that $m(U^a_1) <C_1 a^{2}$ and $m(U^b_1) < C_1 b^2$. There exists $C_2>1$ such that $C_2^{-1}< \|(Df_{a,b})^{-1}\|^{-1} < \|Df_{a,b}\| < C_2$, for any $a$ and $b$ small enough. Using property (c) of $f_{a,b}$ we can easily conclude (\ref{eq.inequalityintegral}).  
\end{enumerate}
Observe that such properties are $C^1$-open. We fix a periodic point $q$, whose orbit remains outside $U^a_1 \cup U^b_2$. Take $f\in \mathrm{Diff}^1_{\omega}(\T^4)$ a diffeomorphisms sufficiently $C^1$-close to $f_{a,b}$ such that the homoclinic class $H(q_{f})$ is the entire manifold $\T^4$, where $q_{f}$ is the continuation of the hyperbolic periodic point $q$. This is possible since $C^1$-generically in $\mathrm{Diff}^1_{\omega}(\T^4)$ the entire manifold is the homoclinic class of any periodic point, see theorem $1.3$ in \cite{bonatticrovisierpseudo}.

We now explain how to obtain trapped plaque families, as in the definition of chain-hyperbolicity. Let $C_2$ be the constant that appeared in the explanation of property (d) of $f_{a,b}$. Fix $r>0$ such that for any $p\in \T^4$, any disc $D(x)$ with radius $r$, center $x$ and tangent to $\mathcal{C}^s$, the set $D(x) \cap U^a_1$ has at most one connected component. Let $\rho = \frac{r}{C_2}$.

\begin{lemma}
\label{lem.contract}
If $a$ is small enough, for any $x\in \T^4$, any disc $D(x)$ tangent to $\mathcal{C}^s$ with radius $\rho$ and centered in $x$, it holds that $f^{-1}(D(x))$ strictly contains a disc of radius $\rho$, centered in $f^{-1}(x)$ and tangent to $\mathcal{C}^s$.
\end{lemma}
\begin{proof}
Let $D(x)$ be such a disc and let $\mu>1$ be a constant such that for any unit vector $v\in \mathcal{C}^s(x)$ and $x\notin U^a_1$, it holds that $\|Df^{-1}(x)v\| > \mu $. This comes from property (c) of $f_{a,b}$ and the fact that $f$ is $C^1$-close to $f_{a,b}$.

Domination implies that $f^{-1}(D(x))$ is tangent to $\mathcal{C}^s$. By our choice of $\rho$, we have that $f^{-1}(D(x)) \cap U^a_1$ has at most one connected component. Since $f^{-1}(D(x))$ is tangent to the cone $\mathcal{C}^s$, there exists a constant $C_3>1$ such that the diameter of the set $f^{-1}(D(x)) \cap U^a_1$ is bounded from above by $C_3a$. 

Let $\gamma$ be a curve minimizing distance between $f^{-1}(x)$ and $\partial f^{-1}(D(x))$. It holds that $l(\gamma) > \rho \|Df^{-1}\|^{-1}> \frac{\rho}{C_2}$, where $l(\gamma)$ is the length of the curve $\gamma$. We split $\gamma$ in two parts: $\gamma_1 = \gamma \cap U^a_1$ and $\gamma_2= \gamma - \gamma_1$. Observe that $l(\gamma_1) < C_3a$. Thus
\[
\frac{l(\gamma_1) }{l(\gamma)} < \frac{C_2C_3 a}{\rho}= Ka.
\]
Thus,
\[
\begin{array}{rcl}
l(f(\gamma)) &= &l(f(\gamma_1)) + l(f(\gamma_2)) <   C_2 l(\gamma_1)+\mu^{-1} l(\gamma_2)\\
 & < &  C_2 Ka  l(\gamma) + \mu^{-1}l(\gamma)= ( C_2 Ka + \mu^{-1}) l(\gamma) < l(\gamma), 
\end{array}
\]
where the last inequality holds for $a>0$ small enough. Observe that $f(\gamma)$ is a curve connecting $x$ to $\partial D(x)$, we conclude that $d(x, \partial D(x))< d(f^{-1}(x), \partial f^{-1}(D(x)))$, so $f^{-1}(D(x))$ contains a disc centered in $f^{-1}(x)$ with radius $\rho$.
\end{proof}

Following the proof of theorem $3.1$ in \cite{buzzifisher}, using lemma \ref{lem.contract} in the place of claim $3.2$ in the same paper, a construction using graph transforms allows us to obtain a plaque family $(\mathcal{W}^{cs}_x)_{x\in \T^4}$ which is trapped for $f$, such that any plaque $\mathcal{W}^{cs}_x$ is a disc of center $x$, radius $\rho$ and is tangent to $\mathcal{C}^s$. Similarly, for $b$ small enough we obtain a plaque family $(\mathcal{W}^{cu}_x)_{x\in \T^4}$, which is trapped for $f^{-1}$.

By taking $a,b$ small enough, we can also suppose that 
\begin{equation}
\label{propertyplaque}
\displaystyle \bigcup_{q_f\in O_f(q_{f})} (\mathcal{W}^{cs}_{q_f} \cup \mathcal{W}^{cu}_{q_f}) \cap (U^a_1 \cup U^b_2) = \emptyset. 
\end{equation}
A standard argument known as the coherence argument (see for instance the argument used in step 2 in the proof of theorem $3.1$ in \cite{buzzifisher}), implies that the plaques $\mathcal{W}^{cs}_{q_f}$ and $\mathcal{W}^{cu}_{q_f}$ are contained in the stable and unstable manifolds of $q_f$, respectively. Thus, conditions $3$ and $4$ in definition \ref{chainhyperbolicity} hold.
 
Property $1$ of $f_{a,b}$ implies that condition $2$ in definition \ref{chainhyperbolicity} is verified. Since $\T^4$ is the homoclinic class of $q_{f}$, we conclude that condition $1$ in definition \ref{chainhyperbolicity} is verified. Therefore, $f$ is chain-hyperbolic.

We conclude that all the conditions in the hypothesis of theorem A are verified, thus there is a $C^1$-neighborhood $\mathcal{U}$ of $f$ such that any diffeormorphism $g\in \mathcal{U}\cap \mathrm{Diff}^2_{\omega}(\T^4)$ is Bernoulli. 

We remark that this construction is not local because the cilinders $U^a_1$ and $U^b_2$, where we made the changes, have a fix size $R$ in one of the directions, either $E^{cs}$ or $E^{cu}$.
  
\bibliographystyle{alpha}

\address

\end{document}